\documentclass{amsart}
\usepackage{amssymb, amsmath}
\usepackage[a4paper,top=3cm,bottom=3cm,left=1.5 cm,right=1.5cm]{geometry}

\usepackage{xcolor}
\usepackage{comment}

\newtheorem{Theorem}{Theorem}
\newtheorem{Definition}[Theorem]{Definition}
\newtheorem{Remark}[Theorem]{Remark}

\newtheorem{Proposition}[Theorem]{Proposition}
\newtheorem{Lemma}[Theorem]{Lemma}
\newtheorem{Corollary}[Theorem]{Corollary}

\newcommand{\F}{F\langle X\rangle}

\DeclareMathOperator{\Id}{Id}
\DeclareMathOperator{\V}{var}

\begin{document}
\title[Codimensions of algebras with pseudoautomorphism and their exponential growth]{Codimensions of algebras with pseudoautomorphism \\ and their exponential growth}

\author{Elena Campedel}
\address{Dipartimento di Ingegneria e Scienze dell'Informazione e Matematica, Universit\`a degli Studi dell'Aquila, Via Vetoio 1, 67100, L'Aquila, Italy}
\email{elena.campedel@univaq.it}

\author{Ginevra Giordani}
\address{Dipartimento di Ingegneria e Scienze dell'Informazione e Matematica, Universit\`a degli Studi dell'Aquila, Via Vetoio 1, 67100, L'Aquila, Italy}
\email{ginevra.giordani@graduate.univaq.it}

\author{Antonio Ioppolo}
\address{Dipartimento di Ingegneria e Scienze dell'Informazione e Matematica, Universit\`a degli Studi dell'Aquila, Via Vetoio 1, 67100, L'Aquila, Italy}
\email{antonio.ioppolo@univaq.it}

\thanks{The authors were supported by GNSAGA-INDAM}



\begin{abstract}
Let $F$ be a fixed field of characteristic zero containing an element $i$ such that $i^2 = -1$. In this paper we consider finite dimensional superalgebras over $F$ endowed with a pseudoautomorphism $p$ and we investigate the asymptotic behaviour of the corresponding sequence of  $p$-codimensions $c_n^p(A),$ $n=1,2, \ldots$. 
First we give a positive answer to a conjecture of Amitsur in this setting: the $p$-exponent $\exp^p(A) = \lim_{n \rightarrow \infty} \sqrt[n]{c_n^p(A)} $ always exists and it is an integer. In the final part we characterize the algebras whose exponential growth is bounded by $2$.
\end{abstract}

\maketitle

\section{\bf{Introduction}}

Let $F$ be a fixed field of characteristic zero containing an element $i$ such that $i^2 = -1$ and consider an associative superalgebra $A = A_0 \oplus A_1$ over $F$. A graded linear map $p$ is a pseudoautomorphism on $A$ if, for any homogeneous elements $a, b \in A_0\cup A_1$, 
$$
p^2(a) = (-1)^{\vert a \vert}a \ \ \ \ \ \mbox{ and } \ \ \ \ \ 
p(ab)=(-1)^{\vert a \vert \vert b \vert} p(a) p(b).
$$

In \cite{Ioppolo2022}, the author proved that pseudoautomorphisms represent the connection link between graded involutions, superinvolutions and pseudoinvolutions; such maps play a prominent role in the setting of Lie and Jordan superalgebras (see for instance \cite{Kac1977, MartinezZelmanov2010, RacineZelmanov2003}) and they have been extensively studied within the theory of polynomial identities (see \cite{CentroneEstradaIoppolo2022, IoppoloMartino2022} and the references therein). 

It is well-known that the study of the polynomial identities satisfied by an ordinary algebra $A$ (with no additional structure) is equivalent to the study of the multilinear ones and an effective way to measure such identities is through the sequence of codimensions $c_n(A),$ $n=1, 2, \ldots,$ of $A$. Recall that if $P_n$ is the space of multilinear polynomials in the non-commuting variables $x_1, \ldots, x_n$ and $\Id(A)$ is the ideal of identities of $A,$ then $c_n(A)=\dim P_n/(P_n\cap \Id(A))$. The asymptotic behavior of this sequence has been extensively
studied  leading to classification results of several
varieties of algebras. The key result in this area says that the
sequence of codimensions of an algebra satisfying a
non-trivial polynomial identity is exponentially bounded (\cite{Regev1972}) and its exponential rate of growth is an integer (\cite{GiambrunoZaicev1998, GiambrunoZaicev1999}).

Let $A$ be a superalgebra with pseudoautomorphism. In this paper we study the $p$-polynomial identities satisfied by $A$ and we investigate the asymptotic behaviour of the corresponding sequence $c_n^p(A)$ of $p$-codimensions. Notice that such a sequence is bounded from above by $4^n  n!$. Nevertheless when $A$ satisfies an ordinary (non trivial) identity, $c_n^p(A)$ is exponentially bounded (see \cite{Giordani}).

Now assume that $A$ has finite dimension over the field $F$. In the first part of the paper we determine the exponential rate of growth of the sequence of
$p$-codimensions, showing that 
$$
\exp^p(A) = \lim_{n \rightarrow \infty} \sqrt[n]{c_n^p(A)} 
$$
exists and it is a non-negative integer, called the $p$-exponent of $A$.
Moreover $\exp^p(A)$ can be explicitly computed and it turns out to be equal to the dimension of a suitable semisimple $p$-subalgebra of $A$.

The last part of the paper is devoted to the characterization of those algebras whose $p$-exponent is bounded by $2$ (see also \cite{GiambrunoIoppoloLaMattina2022, GiambrunoZaicev2000, Ioppolo2021, IoppoloMartino2019}). If the $p$-exponent of an algebra $A$ is bounded by $1,$ it is equivalent to say that the $p$-codimensions are polynomially bounded and   that the variety generated by $A$ does not contain the group algebra of $\mathbb{Z}_2$ and the algebra of $2\times 2$ upper triangular matrices with suitable pseudoautomorphisms (see \cite{Giordani}). These are the only algebras generating minimal varieties of $p$-exponent $2.$ 

Finally, new results concerning  $p$-algebras generating varieties  of minimal $p$-exponent $>2$ will be obtained. 

It is important to highlight that the starting point in the proof of all results of this paper is the Wedderburn-Malcev decomposition of a finite dimensional $p$-algebra based on the classification of the simple ones given in \cite{IoppoloLaMattina}.

\section{\bf{Superalgebras with pseudoautomorphism}}

Throughout this paper $F$ will denote a field of characteristic zero containing an element $i$ such that $i^2 = -1$ and $A=A_{0}\oplus A_{1}$ an associative superalgebra (an algebra graded by $\mathbb{Z}_2$, the cyclic group of order $2$) over $F$ endowed with a pseudoautomorphism $p$. We say that a linear map $p \colon A \rightarrow A$ is a pseudoautomorphism if it preserves the grading (graded map) and for any elements $a, b \in A_0\cup A_1$,
$$
p^2(a) = (-1)^{\vert a \vert}a \ \ \ \ \ \mbox{ and } \ \ \ \ \ 
p(ab)=(-1)^{\vert a \vert \vert b \vert} p(a) p(b).
$$
 Here $\vert c \vert=0$ or $1$ denotes the homogeneous degree of $c \in A_0\cup A_1.$ 

If $A$ is a superalgebra with pseudoautomorphism we shall call it simply a $p$-algebra.

In case $A$ is a finite dimensional algebra, its structure is known from a generalization of Wedderburn-Malcev's theorem proved in \cite[Theorem 3]{Giordani}. Before stating it, recall that an ideal (subalgebra) $I$ of $A$ is a $p$-ideal (subalgebra) of $A$ if it is a graded ideal (subalgebra) and $p(I)=I$. The algebra $A$ is a simple $p$-algebra if $A^2\neq 0$ and $A$ has no non-trivial $p$-ideals.

\begin{Theorem}  \label{WM}
Let $A$ be a finite dimensional $p$-algebra. Then there exists a semisimple $p$-subalgebra $B$ such that
$$
A=B+J = B_1 \oplus \cdots \oplus B_k + J,
$$
where $J,$ the Jacobson radical of $A,$ is  a $p$-ideal of $A$ and $B_1,  \ldots, B_k$ are simple $p$-algebras.
\end{Theorem}

Since the classification of the simple $p$-algebras is known, the above result can be further refined. 

To this end, consider the following simple superalgebras:
\begin{itemize}
\item[-] $ Q(n) = M_n(F) \oplus c M_n(F)$, where $M_n(F)$ is the algebra of $n \times n$ matrices over $F$ and $c^2 = 1;$
\vspace{0.1 cm}
\item[-] $M_{k,h}(F)$, the algebra of $n \times n$ matrices, $n = k+h$, $ k \geq h \geq 0$, with the following $\mathbb{Z}_2$-grading
$$
M_{k,h}(F) = \left \{ \begin{pmatrix}
K & 0 \\
0 & H
\end{pmatrix}
  \mid  K \in M_k(F), \ H \in M_h(F) \right \} \oplus
\left \{ \begin{pmatrix}
0 & R \\
S & 0
\end{pmatrix}
  \mid  R \in M_{k \times h}(F), \ S \in M_{h \times k}(F) \right \}. 
$$ 
\end{itemize}

Given a superalgebra $B$, let  $\bar{B}$ denote the superalgebra with the same graded vector space structure as $B$ and product $\circ$ given on homogeneous elements $a,b$ by the formula
$$
a \circ b := (-1)^{\vert a \vert \vert b \vert} ab.
$$

Two $p$-algebras $(A, p)$ and $(A', p')$ are isomorphic if there exists an isomorphism of superalgebras  $\tau \colon A \rightarrow A'$ such that $\tau(p(a))=p'(\tau(a)),$ for any $a\in A$.

Now we are ready to state the classification theorem of the finite dimensional simple $p$-algebras (see \cite{IoppoloLaMattina}). 

\begin{Theorem}  \label{phi super simple algebras}
Assume that the field $F$ is also algebraically closed. A finite dimensional $p$-simple superalgebra $A$ over $F$ is isomorphic to one of the following:
\begin{enumerate}

\item $M_{k,h}(F)$ endowed with the pseudoautomorphism
			\[p \left(\begin{pmatrix}
				K & R \\
				S & H \\
			\end{pmatrix} \right)=\begin{pmatrix}
				PKP & \pm iPRQ \\
				\pm iQSP & QHQ \\
			\end{pmatrix},\] 
 $P=\begin{pmatrix}
				I_{k_1} & 0 \\
				0 & -I_{k_2} \\
			\end{pmatrix}, Q=\begin{pmatrix}
				I_{h_1} & 0 \\
				0 & -I_{h_2} \\
			\end{pmatrix}$, $I_j$'s are identity matrices, $k=k_1+k_2$, $h=h_1+h_2$, $k_1 \ge k_2$, $h_1 \ge h_2$.
\vspace{0.1 cm}

\item $M_{k,k}(F)$ endowed with the pseudoautomorphism $p$ given by
			$\displaystyle p \left( \begin{pmatrix}
				K & R \\
				S & H \\
			\end{pmatrix} \right)=\begin{pmatrix}
				H & iS \\
				iR & K \\
			\end{pmatrix}$. 

\vspace{0.1 cm}
			
\item $M_{k,h}(F)\oplus\overline{M_{k,h}(F)}$ with the pseudoautomorphism $pex$ given by
$\displaystyle pex (a,b) = \left((-1)^{\vert(a,b) \vert}b,a \right)$.

\vspace{0.1 cm}
		
\item $Q(n) \oplus \overline{Q(n)}$, with the pseudoautomorphism $pex$ defined above.
\vspace{0.25 cm}
					
\item $Q(n)$ endowed with the following pseudoautomorphism
$$
p(a+cb) = f(a) \pm icf(b),
$$
where $f$ is an automorphism of order $\le 2$ on $M_n(F)$.

\end{enumerate}
\end{Theorem}                                               

\section{{\bf The $p$-exponent}} \label{Section 3}
Now we are interested in studying the $p$-algebras in the context of the theory of polynomial identities.  

If $A$ is a $p$-algebra, since $\mbox{char}\,F=0$ and there exists $i \in F$ such that $i^2 = -1$, we can write 
$$
A=A_0^+ \oplus A_0^- \oplus A_1^i \oplus A_1^{-i}.
$$
Here $A_0^+=\{a \in A_0 \mid p(a)=a\}$ and $A_0^-=\{a \in A_0 \mid p(a)=-a\}$ are the sets of symmetric and skew elements of $A_0$, respectively, and $A_1^i=\{a \in A_1 \mid p(a)=ia\}$ and $A_1^{-i}=\{a \in A_1 \mid p(a)=-ia\}$ are the sets of the so-called $i$-symmetric and $i$-skew elements of $A_1$, respectively.

One can define in a natural way a pseudoautomorphism on the free associative algebra $\F$ on a countable set $X$ over $F.$ We write $X$ as the union of two disjoint infinite sets $Y$ and $Z$, requiring that their elements are of homogeneous degree 0 and 1, respectively. Then each set is written as the disjoint union of two other infinite sets of symmetric and skew elements and of $i$-symmetric and $i$-skew elements, respectively. The free superalgebra with pseudoautomorphism is denoted $F\langle Y\cup Z, p \rangle $ and we write
$$
F\langle Y\cup Z, p \rangle = F\langle y^+_1, y^-_1, z^+_1, z^-_1, y^+_2, y^-_2, z^+_2, z^-_2, \ldots\rangle,
$$
where $y_i^+$ stands for a (even) symmetric variable, $y_i^-$ for a (even) skew variable, $z_i^+$ for a (odd) $i$-symmetric variable and $z_i^-$ for a (odd) $i$-skew variable.

A polynomial $f\in F\langle Y\cup Z, p \rangle$ is a $p$-polynomial identity of $A$ (or simply a $p$-identity), and we write $f\equiv 0,$ if
it vanishes for all substitutions $y^{\pm} \mapsto a^{\pm}\in A_{0}^{\pm},$ $z^{\pm} \mapsto b^{\pm}\in A_{1}^{\pm i}$. Let  $\Id^{p}(A)$ denote the set of all $p$-identities of $A$. Clearly it is an ideal of $F\langle Y\cup Z, p \rangle$ invariant under  all endomorphisms of the free superalgebra commuting with the pseudoautomorphism $p$. As in the ordinary case, it is easily seen that in  characteristic zero, every $p$-identity is equivalent to a system of multilinear $p$-identities. Hence if
we denote by
$$
P^{p}_n=\mbox{span}_F\{w_{\sigma(1)} \cdots w_{\sigma(n)}\mid \sigma
\in S_n, \ w_i \in \{ y^+_i, y^-_i, z^+_i, z^-_i \},   \ i=1, \ldots, n\}
$$
the space of  multilinear polynomials of degree $n$ in the variables
$y^+_i$, $y^-_i$, $z^+_i$, $z^-_i$, $i=1, \ldots, n,$ the study of $\Id^{p}(A)$ is equivalent to
the study of $P^{p}_n\cap \Id^{p}(A)$, for all $n\geq 1$. The $n$-th $p$-codimension of $A$ is the non-negative integer
$$
c^{p}_n(A)=\dim_F\frac{P^{p}_n}{P^{p}_n\cap \Id^{p}(A)},\ n\geq 1.
$$
If $A$  satisfies an ordinary polynomial identity, it was proved that $c^{p}_n(A)$, $n=1,2, \ldots, $ is exponentially bounded (\cite{Giordani}).

\medskip

The first aim of this paper is to determine the exponential rate of growth of the sequence of $p$-codimensions of a finite dimensional $p$-algebra. We start with the following definition.

\begin{Definition} \label{definition of admissible and reduced algebra}
Let $A =B_1 \oplus \cdots \oplus B_k + J $ be a finite dimensional $p$-algebra and let $C_1, \ldots, C_h$ be distinct simple $p$-subalgebras of $A$ from $\{B_1, \ldots, B_k\}$. The  $p$-algebra $C=C_1\oplus \cdots\oplus C_h$ is called
 $p$-admissible if 
 $$C_{1}J\cdots J C_{h-1}JC_h \ne 0.$$
\end{Definition}

Now we can define the following integer.

\begin{Definition} \label{definition of d}
Let $A$ be a finite dimensional $p$-algebra. We set
$$
d = d(A) := \max \left \{ \dim_F C : \ C \ \mbox{is an admissible p-subalgebra of } A \right \}.
$$
\end{Definition}

The role of such an integer in the description of the asymptotic behavior of the $p$-codimensions is explained in the following result.



\begin{Theorem} \label{the exponent}
Let $A$ be a finite dimensional $p$-algebra over $F$ and consider the integer $d$ of Definition \ref{definition of d}. Then there exist constants $a_1 >0$ and $a_2, b_1, b_2$ such that
$$
a_1 n^{b_1} d^n \leq c_n^p(A) \leq a_2 n^{b_2} d^n.
$$ 
Hence the $p$-exponent of $A$, $\exp^p(A) = \lim_{n \rightarrow \infty} \sqrt[n]{c_n^p(A)} $ exists and it is a non-negative integer. 

\end{Theorem}
\begin{proof}
Since the $p$-codimensions do not change by extending the ground field, we may assume that the field $F$ is algebraically closed. Now the result can be proved, with the necessary changes, by following word by word the proof given in \cite{Ioppolo2018} in the setting of superalgebras with superinvolution.
\end{proof}

As an immediate consequence we get the following.

\begin{Corollary} 
Under the hypotheses of Theorem \ref{the exponent}, the sequence
$c_n^{p}(A)$, $n=1, 2, \ldots$, either is polynomially bounded $($i.e., $\exp^p(A)\leq 1)$ or it grows exponentially $($i.e., $\exp^p(A)\ge 2)$.
\end{Corollary}

In \cite{Giordani}, the second author described the varieties of $p$-algebras of polynomial growth by giving a finite list of $p$-algebras to be excluded from the variety. Recall that the growth of a variety $\mathcal{V}$ of $p$-algebras is defined as the growth of the sequence of $p$-codimensions of any algebra $A$ generating $\mathcal{V}$, i.e., $\mathcal{V}= \V^p(A)$. Then we say that $\mathcal{V}$ has polynomial growth if $c_n^p(\mathcal{V})$
 is polynomially bounded and $\mathcal{V}$ has
almost polynomial growth if $c_n^p(\mathcal{V})$ is not polynomially bounded but every proper subvariety of $\mathcal{V}$ has polynomial growth.



Let $F\oplus F$ be the two-dimensional commutative algebra. We can see it as a superalgebra with trivial grading or with the natural grading
$$
F \oplus F = F(1,1) \oplus F(1,-1).
$$
We consider the following three superalgebras with pseudoautomorphism:
\begin{itemize}
    \item $D$, the algebra $F\oplus F$ with trivial grading and  pseudoautomorphism 
    $
    ex(a,b) = (b,a).
    $
\vspace{0.1 cm}
    
    \item $D^i$, the algebra $F\oplus F$ with natural grading and pseudoautomorphism 
    $
    ex_i(a,b) = i^{\vert(a,b) \vert}(b,a).
    $
    \vspace{0.1 cm}
    
    \item $D^{-i}$, the algebra $F\oplus F$ with natural grading and pseudoautomorphism 
    $
    ex_{-i}(a,b) = (-i)^{\vert(a,b) \vert}(b,a).
    $
\end{itemize}

Now consider the algebra $
UT_2(F) = \left \{ \begin{pmatrix}
  a & b\\ 
  0 & c
\end{pmatrix} \mid a,b,c \in F \right \}
$ of $2 \times 2$ upper-triangular matrices.
We can see it as a superalgebra with trivial grading or with the natural grading 
$$ 
UT_2(F) = \left \{ \begin{pmatrix}
  a & 0\\ 
  0 & c
\end{pmatrix} \mid a,c \in F \right  \} \oplus \left \{ \begin{pmatrix}
  0 & b\\ 
  0 & 0
\end{pmatrix}    \mid b \in F \right \}. 
$$

We consider the following four superalgebras with pseudoautomorphism:
\begin{itemize}
    \item $UT_2$, the algebra $UT_2(F)$ with trivial grading and trivial pseudoautomorphism.
    \vspace{0.2 cm}

    \item $UT_2^-$, the algebra $UT_2(F)$ with trivial grading and pseudoautomorphism:
    $
    p \left( \begin{pmatrix}
					a & b \\
					0 & c \\
				\end{pmatrix} \right) = \begin{pmatrix}
					a & -b \\
					0 & c \\
				\end{pmatrix}.
    $
    \vspace{0.1 cm}
    
\item $UT_2^i$, the algebra $UT_2(F)$ with natural grading and pseudoautomorphism:
    $
    p \left( \begin{pmatrix}
					a & b \\
					0 & c \\
				\end{pmatrix} \right) = \begin{pmatrix}
					a & ib \\
					0 & c \\
				\end{pmatrix}.
    $
    \vspace{0.2 cm}

\item $UT_2^{-i}$, the algebra $UT_2(F)$ with natural grading and pseudoautomorphism:
    $
    p \left( \begin{pmatrix}
					a & b \\
					0 & c \\
				\end{pmatrix} \right) = \begin{pmatrix}
					a & -ib \\
					0 & c \\
				\end{pmatrix}.
    $  
    \vspace{0.2 cm}
\end{itemize}

These $p$-algebras characterize the varieties of polynomial growth as stated in the following result (\cite[Theorem 9]{Giordani}).

\begin{Theorem}  \label{cn(A) pol bounded}
Let $A$ be a finite dimensional $p$-algebra over $F$. The sequence $c_n^{p}(A)$, $n=1, 2, \ldots,$ is polynomially bounded $($i.e., $\exp^p(A)\leq 1)$  if and only if  $UT_2$, $UT_2^{-}$, $UT_2^{i}$, $UT_2^{-i}$, $D$, $D^i$, $D^{-i}  \notin \V^{p}(A). $ 
\end{Theorem}

Given two $p$-algebras $A$ and $B$ we say that they are equivalent if $\Id^p(A)=\Id^p(B).$
\begin{Corollary}\cite{Giordani}  \label{UT2s are APG} 
The algebras $UT_2$, $UT_2^{-}$, $UT_2^{i},$  $UT_2^{-i}$, $D$, $D^i$, $D^{-i}$ are the only finite dimensional $p$-algebras, up to equivalence, generating varieties of almost polynomial growth.
\end{Corollary}

Now we recall that a variety  $\mathcal V$ of $p$-algebras
is minimal with respect to the $p$-exponent if for any proper subvariety  $\mathcal U,$ 
we have that $\exp^{p}(\mathcal{V})>\exp^{p}(\mathcal{U}).$ 
Here the $p$-exponent of a variety is the $p$-exponent of a generating algebra.
Since the above algebras have $p$-exponent equal to 2, by using this definition we get the following result.

\begin{Corollary} 
    The algebras $UT_2$, $UT_2^{-}$, $UT_2^{i}$, $UT_2^{-i}$, $D$, $D^i$, $D^{-i}$ are the only finite dimensional $p$-algebras, up to equivalence, generating minimal varieties of $p$-exponent $2$.
   \end{Corollary}
\medskip

\section{\bf{Characterizing algebras with $p$-exponent bounded by 2}}

We start this section by constructing a suitable finite list of $p$-algebras generating varieties of $p$-exponent $\ge 2$. 

Let us consider the following simple $p$-algebras: 

\begin{itemize}
\vspace{0.1 cm}

\item[$\bullet$] $C_1$, the superalgebra $M_{2,0}(F)$ with trivial pseudoautomorphism $\textit{id}$;
\vspace{0.2 cm}

\item[$\bullet$] $C_2$, the superalgebra $M_{2,0}(F)$ with  pseudoautomorphism $p$ given by
$
p \left( \begin{pmatrix}
  a & b\\ 
  c & d
\end{pmatrix} \right) = \begin{pmatrix}
  a & -b\\ 
  -c & d
\end{pmatrix};
$
\vspace{0.2 cm}

\item[$\bullet$] $C_3$, the superalgebra $M_{1,1}(F)$  with  pseudoautomorphism $p$ given by
$
p \left( \begin{pmatrix}
  a & b\\ 
  c & d
\end{pmatrix} \right) = \begin{pmatrix}
  a & ib\\ 
  ic & d
\end{pmatrix};
$

\vspace{0.2 cm}

\item[$\bullet$] $C_4$, the superalgebra $M_{1,1}(F)$  with  pseudoautomorphism $p$ given by
$
p \left( \begin{pmatrix}
  a & b\\ 
  c & d
\end{pmatrix} \right) = \begin{pmatrix}
  a & -ib\\ 
  -ic & d
\end{pmatrix};
$

\vspace{0.2 cm}

\item[$\bullet$] $C_5$, the superalgebra $M_{1,1}(F)$  with  pseudoautomorphism $p$ given by
$
p \left( \begin{pmatrix}
  a & b\\ 
  c & d
\end{pmatrix} \right) = \begin{pmatrix}
  d & ic\\ 
  ib & a
\end{pmatrix};
$
\vspace{0.2 cm}

\item[$\bullet$] $C_6$, the superalgebra $Q(1) \oplus \overline{Q(1)}$ with the pseudoautomorphism  $pex(a,b) = \left( (-1)^{\vert(a,b)\vert} b,a \right)$.
\vspace{0.2 cm}
\end{itemize}

By Theorem \ref{the exponent} we get the following result. 

\begin{Remark}\label{exponent of C1 C2 C3 C4}
For any $i = 1, \ldots, 6$,  $\exp^p(C_i) = 4.$
\end{Remark} 

The above $p$-algebras allow us to prove the following lemma. 

\begin{Lemma} \label{remark}
Let $B$ be a simple $p$-algebra  with $\dim_F B \geq 4$. Then $C_j \in \V^{p}(B)$, for some $j \in 	\left \{ 1,\ldots, 6 \right \}$.
\end{Lemma}
\begin{proof}
We shall prove the lemma by constructing   a $p$-subalgebra of $B$ isomorphic to $C_j$ for some $j \in \left \{ 1,\ldots, 6 \right \}$.   

\medskip

{\bf Case 1.} $B =\left ( M_{k,h}(F), p \right )$ with
$
p\left(\begin{pmatrix}
  K & R\\ 
  S & H
\end{pmatrix}\right) = \begin{pmatrix}
  PKP & \pm i PRQ\\ 
  \pm i QSP & QHQ
\end{pmatrix}   
$, $P=\begin{pmatrix}
				I_{k_1} & 0 \\
				0 & -I_{k_2} \\
			\end{pmatrix}, Q=\begin{pmatrix}
				I_{h_1} & 0 \\
				0 & -I_{h_2} \\
			\end{pmatrix}. 
\vspace{0.2 cm}
$

Suppose first that $h = 0$. This means that the superalgebra $B$ has trivial grading and the pseudoautomorphism $p$ is just a graded automorphism. If $k_2 = 0$, $p$ is the identity map and it is immediate to see that the $p$-subalgebra of $B$ generated by the elements $ a_1 = e_{1,1}$, $a_2 = e_{1,k_1}$, $a_3 = e_{k_1,1}$, $a_4 = e_{k_1,k_1}$ is isomorphic to $C_1$. Now assume $k_2> 0$. We easily get that the $p$-subalgebra $C'$ of $B$ generated by the elements $a_1 = e_{1,1}$,  $a_2 = e_{k_1+1, k_1+1}$, $a_3 = e_{1, k_1+1}$, $a_4 = e_{k_1+1, 1}$ is isomorphic to $C_2$ through the isomorphism  $f \colon C' \rightarrow C_2$, given by
$$
f(a_1) = e_{1,1}, \ \ \ 
f(a_2) = e_{2,2}, \ \ \ 
f(a_3) = e_{1,2}, \ \ \ 
f(a_4) = e_{2,1}. \ \ \
$$

We are left to deal with the case $h > 0$. Let $C'$ be the $p$-subalgebra of $B$ generated by the elements $a_1 = e_{1,1}$,  $a_2 = e_{k+1, k+1}$, $a_3 = e_{1, k+1}$, $a_4 = e_{k+1, 1}$. The linear map $f$ given by
$$
f(a_1) = e_{1,1}, \ \ \ 
f(a_2) = e_{2,2}, \ \ \ 
f(a_3) = e_{1,2}, \ \ \ 
f(a_4) = e_{2,1}, \ \ \
$$
is an isomorphism of $p$-algebras between $C'$ and $C_3$ or $C_4$, according to the sign ($\pm i$) of the pseudoautomorphism $p$.

\medskip

{\bf Case 2.} $B =\left ( M_{k,k}(F), p \right )$ with
$
p \left( \begin{pmatrix}
  K & R\\ 
  S & H
\end{pmatrix} \right) = \begin{pmatrix}
  H & i S\\ 
   i R & K
\end{pmatrix}   
$. 
\vspace{0.2 cm}

Let $C'$ be the $p$-subalgebra of $B$ generated by $a_1 = e_{1,1}$,  $a_2 = e_{k+1, k+1}$, $a_3 = e_{1, k+1}$, $a_4 = e_{k+1, 1}$. Hence we obtain an isomorphism of $p$-algebras between $C'$ and $C_5$ via the linear map $f \colon C' \rightarrow C_5$ given by
$$
f(a_1) = e_{1,1}, \ \ \ 
f(a_2) = e_{2,2}, \ \ \ 
f(a_3) = e_{1,2}, \ \ \ 
f(a_4) = e_{2,1}. \ \ \
$$

\medskip

{\bf Case 3.} $ B=\left ( M_{k,h}(F) \oplus \overline{M_{k,h}(F)}, pex \right),$ $k\ge h \ge 0.$ 

Suppose first $k=h=1.$
The $p$-subalgebra $C'$ generated by the elements $a_1 = (e_{1,1} + e_{2,2}, 0)$, $a_2 = (e_{1,2} + e_{2,1}, 0)$, $a_3 = (0, e_{1,1} + e_{2,2})
$, $a_4 = (0, e_{1,2} + e_{2,1})$ is isomorphic to $C_6=(F\oplus cF)\oplus \overline{(F\oplus cF)}$ through the isomorphism of $p$-algebras:
 $f \colon C' \rightarrow C_6$, given by
$$
f(a_1) = (1,0), \ \ \ 
f(a_2) = (c1,0), \ \ \ 
f(a_3) = (0,1), \ \ \ 
f(a_4) = (0,c1).\ \ \
$$

Now assume $k > 1$.
In this case, we get that the $p$-subalgebra of $B$
generated by  $a_1 = (e_{1,1}, e_{1,1})$, $a_2 = (e_{2,2}, e_{2,2})$, $a_3 = (e_{1,2}, e_{1,2})$, $a_4 = (e_{2,1}, e_{2,1})$ is isomorphic to  $C_1$.

\medskip

{\bf Case 4.} $ B =\left ( Q(n) \oplus \overline{Q(n)}, pex \right) $.

If $n = 1$, then $B = C_6$ and there is nothing to prove. Now, let $n > 1$.
Then $C_1$ is isomorphic to the $p$-subalgebra of $B$ generated  by the elements $a_1 = (e_{1,1}, e_{1,1})$, $a_2 = (e_{2,2}, e_{2,2})$, $a_3 = (e_{1,2}, e_{1,2})$ and $a_4 = (e_{2,1}, e_{2,1})$.

\medskip

{\bf Case 5.} $B = Q(n)$ with pseudoautomorphism
$
p(a+cb) = f(a) \pm i cf(b)
$,
$f$ automorphism of order $\le 2$ on $M_n(F)$.

Since $\dim_F B \ge 4$, we have that $n \ge 2$. Hence we can consider the $p$-subalgebra $C'$ of $B$ generated by $a_1 = e_{1,1}$, $a_2 = e_{2,2}$, $a_3 = e_{1,2}$ and $a_4 = e_{2,1}$. $C'$ has trivial grading and induced pseudoautomorphism $p$. Since $f$ is an order $2$ automorphism of $M_n(F)$, it is well-known that it acts trivially on $M_2(F)$ or 
$$
f \left( \begin{pmatrix}
  a & b\\ 
  c & d
\end{pmatrix} \right) = \begin{pmatrix}
  a & -b\\ 
  -c & d
\end{pmatrix}.
$$
According to the action of $f$, we get a $p$-algebras isomorphism between $C'$ and $C_1$ or $C_2$ and we are done.
\end{proof}

Next we need to consider some suitable $\mathbb{Z}_2$-gradings and pseudoautomorphisms on the algebra $UT_3$ of $3 \times 3$ upper triangular matrices. Recall that an arbitrary triple $(g_1, g_2, g_{3})$ of elements of $\mathbb{Z}_2$ defines an elementary $\mathbb{Z}_2$-grading on $UT_{3}$ by setting:
$$
(UT_{3})_0=\mbox{span}\{e_{i,j} \mid g_i+g_j=0 \hspace{0.1 cm}( \mbox{mod} \hspace{0.1 cm}2) \} \ \mbox{and} \
(UT_{3})_1=\mbox{span}\{e_{i,j}  \mid g_i+g_j=1 \hspace{0.1 cm}( \mbox{mod} \hspace{0.1 cm}2)\}.
$$

On $UT_3$ we can define the following automorphisms (of order $\leq 2$):
$$
id \left( \begin{pmatrix}
  a & b & c \\ 
  0 & d & e \\
  0 & 0 & f
\end{pmatrix} \right) = 
\begin{pmatrix}
  a & b & c \\ 
  0 & d & e \\
  0 & 0 & f
\end{pmatrix},
\ \ \ \ \ \ \ \ 
{\varphi}_1 \left( \begin{pmatrix}
  a & b & c \\ 
  0 & d & e \\
  0 & 0 & f
\end{pmatrix} \right) = 
\begin{pmatrix}
  a & -b & -c \\ 
  0 & d & e \\
  0 & 0 & f
\end{pmatrix},
$$

$$
\varphi_2 \left( \begin{pmatrix}
  a & b & c \\ 
  0 & d & e \\
  0 & 0 & f
\end{pmatrix} \right) = 
\begin{pmatrix}
  a & b & -c \\ 
  0 & d & -e \\
  0 & 0 & f
\end{pmatrix},
\ \ \ \ \ \ \ \ 
\varphi_3 \left( \begin{pmatrix}
  a & b & c \\ 
  0 & d & e \\
  0 & 0 & f
\end{pmatrix} \right) = 
\begin{pmatrix}
  a & -b & c \\ 
  0 & d & -e \\
  0 & 0 & f
\end{pmatrix}.
     $$
\vspace{0.2 cm}

If $UT_3$ is endowed with trivial grading, the above automorphisms can be seen as pseudoautomorphisms. 

\medskip

Given any superalgebra $A= A_0 \oplus A_1$, one can consider the following pseudoautomorphism (recall that $i^2 = -1$) 
\begin{align*}
  p \colon A_0 \oplus A_1  &\to A_0 \oplus A_1\\
  a_0 + a_1 &\mapsto a_0 + i a_1.
\end{align*}
Notice that in case the superalgebra has trivial grading, the pseudoautomorphism $p$ is actually the identity map.

According to the result of \cite{IoppoloMartino}, it is not difficult to see that the composition between $p$ and a graded automorphism of order $\le 2$ on $UT_3$ is a pseudoautomorphism of $UT_3$. Hence we have the following $p$-algebras:
\begin{itemize}
\vspace{0.2 cm}
    \item $C_7,$ with trivial grading  and trivial pseudoautomorphism id;
    \vspace{0.2 cm}

    \item $C_8,$ with trivial grading  and pseudoautomorphism $\varphi_1$;
    \vspace{0.2 cm}

    \item $C_9,$ with trivial grading  and pseudoautomorphism $\varphi_2$;
    \vspace{0.2 cm}

    \item $C_{10},$ with trivial grading  and pseudoautomorphism $\varphi_3$;
    \vspace{0.2 cm}

    \item $C_{11},$ with  grading induced by $(0,0,1)$ and pseudoautomorphism $p$;
    \vspace{0.2 cm}

    \item $C_{12},$ with  grading induced by $(0,0,1)$ and pseudoautomorphism $p \circ \varphi_1$;
    \vspace{0.2 cm}

    \item $C_{13},$ with grading induced by $(0,0,1)$ and pseudoautomorphism $p \circ \varphi_2$;
    \vspace{0.2 cm}

    \item $C_{14},$ with grading induced by $(0,0,1)$ and pseudoautomorphism $p \circ \varphi_3$;
    \vspace{0.2 cm}

    \item $C_{15},$ with  grading induced by $(0,1,1)$ and pseudoautomorphism $p$;
    \vspace{0.2 cm}

    \item $C_{16},$ with  grading induced by $(0,1,1)$ and pseudoautomorphism $p \circ \varphi_1$;
    \vspace{0.2 cm}

    \item $C_{17},$ with grading induced by $(0,1,1)$ and pseudoautomorphism $p \circ \varphi_2$;
    \vspace{0.2 cm}

    \item $C_{18},$ with  grading induced by $(0,1,1)$ and pseudoautomorphism $p \circ \varphi_3$;
    \vspace{0.2 cm}

    \item $C_{19},$ with  grading induced by $(0,1,0)$ and pseudoautomorphism $p$;
    \vspace{0.2 cm}

    \item $C_{20},$ with  grading induced by $(0,1,0)$ and pseudoautomorphism $p \circ \varphi_1$;
    \vspace{0.2 cm}

    \item $C_{21},$ with grading induced by $(0,1,0)$ and pseudoautomorphism $p \circ \varphi_2$;
    \vspace{0.2 cm}

    \item $C_{22},$ with  grading induced by $(0,1,0)$ and pseudoautomorphism $p \circ \varphi_3$.
    \vspace{0.2 cm}
    
\end{itemize}

\begin{Remark}\label{exponent of C5 ... C20}
For $j = 7, \ldots, 22$, we have that $\exp^p(C_j) = 3$.
\end{Remark} 
\begin{proof}
All the $p$-algebras $C_j$ have the same Wedderburn-Malcev decomposition:
$$
C_j = A_1 \oplus A_2 \oplus A_3 + J, 
$$
where $A_1 = Fe_{1,1}$, $A_2 = Fe_{2,2}$, $A_3 = Fe_{3,3}$ and $J = Fe_{1,2} \oplus Fe_{1,3} \oplus Fe_{2,3}$. Since $ A_1 J A_2 J A_3 \neq 0,$  $A_1 \oplus A_2 \oplus A_3$ is  a maximal dimensional $p$-admissible subalgebra and the result follows by Theorem \ref{the exponent}.
\end{proof}

The above $p$-algebras allow us to prove the following lemma.

\begin{Lemma} \label{C5 ... C20 in var(A)}
Let $A = B_1 \oplus \cdots \oplus B_k + J$ be a finite dimensional $p$-algebra over an algebraically closed field $F$ of characteristic zero. If there exist three distinct simple components $B_{i_1} \cong B_{i_2} \cong B_{i_3} \cong F$ such that $B_{i_1} J B_{i_2} J B_{i_3} \neq 0$, then $C_j \in \V^p(A)$, for some $j \in \{ 7, \ldots, 22 \}$.
\end{Lemma}
\begin{proof}
Let $e_1, e_2, e_3$ be the unit elements of $B_{i_1}, B_{i_2}, B_{i_3}$, respectively. Then $e_l^2 = e_l^p = e_l \in (B_{i_l})_0$ and $e_r e_s = \delta_{rs} e_r$, for $r,s, l = 1,2,3$. Since $B_{i_1} J B_{i_2} J B_{i_3} \neq 0$ then  $e_1 J e_2 J e_3 \neq 0$. So we may assume that there exist homogeneous (symmetric, skew, $i$-symmetric or $i$-skew) elements $j,j' \in J$ such that
$$
e_1 j e_2 j' e_3 \ne 0.
$$
 Consider the $p$-subalgebra $U$ of $A$ linearly generated by
$$
e_1, \ \ \ e_2, \ \ \ e_3, \ \ \ 
e_1 j e_2, \ \ \ e_2 j' e_3,  \ \ \ 
e_1 j e_2 j' e_3.
$$
The linear map $f \colon U \rightarrow UT_3$, defined by 
$$
f(e_1) = e_{1,1}, \ \ f(e_2) = e_{2,2}, \ \ f(e_3) = e_{3,3}, \ \ 
f(e_1 j e_2) = e_{1,2}, \ \ 
f(e_2 j' e_3) = e_{2,3}, \ \ 
f(e_1 j e_2 j' e_3) = e_{1,3},
$$
is an isomorphism of algebras. Now, by taking into account the homogeneous degrees of $j$ and $j'$ and their symmetry with respect to the pseudoautomorphism $p$, we get an isomorphism of $p$-algebras between $U$ and $C_j$, for some $j = 7, \ldots, 22$.
\end{proof}

Let us consider the algebra
\vspace{0.1 cm}
$$
M = \left \{ \begin{pmatrix} 
a+\alpha a' & e+\alpha e' & 0 & 0 \\
0 & b+\alpha b' & 0 & 0 \\
0 & 0 & c+\alpha c' & 0 \\
0 & 0 & f+\alpha f' & d+\alpha d' 
\end{pmatrix} \ \mid \ a, a', b, b', c, c', d, d', e, e', f, f' \in F, \ \alpha^2 = 1 \right \}
$$
and the automorphism $\dagger$ on it given by
$$
\dagger \left(
\begin{pmatrix} 
a+\alpha a' & e+\alpha e' & 0 & 0 \\
0 & b+\alpha b' & 0 & 0 \\
0 & 0 & c+\alpha c' & 0 \\
0 & 0 & f+\alpha f' & d+\alpha d' 
\end{pmatrix} \right) = \begin{pmatrix} 
d+  \alpha d' & f+  \alpha f' & 0 & 0 \\
0 & c+  \alpha c' & 0 & 0 \\
0 & 0 & b+  \alpha b' & 0 \\
0 & 0 & e+  \alpha e' & a+  \alpha a' 
\end{pmatrix}.
$$
We denote by $M_1$ the algebra $M$ such that $a' = b' = c' = d' = e' = f' = 0$, endowed with trivial grading and pseudoautomorphism $\dagger$. Instead we use the symbol $M_2$ in case the grading is the elementary one induced by $(0,1,1,0)$ and the pseudoautomorphism is $p \circ \dagger$. We need the  following $p$-algebras:
\begin{itemize}
\vspace{0.1 cm}
    \item $C_{23}$, the subalgebra of $M_1$ with $b = c$;
\vspace{0.1 cm}

\item $C_{24}$, the subalgebra of $M_1$ with $a = d$;
\vspace{0.1 cm}

    \item $C_{25}$, the subalgebra of $M_2$ with $b=c$;
\vspace{0.1 cm}

    \item $C_{26}$, the subalgebra of $M_2$ with $a = d$.
\vspace{0.1 cm}
\end{itemize}

Now notice that the algebra $M$ can be seen as a superalgebra also with the following grading:
\vspace{0.1 cm}
$$
\left \{ \begin{pmatrix} 
a & e & 0 & 0 \\
0 & b & 0 & 0 \\
0 & 0 & c & 0 \\
0 & 0 & f & d 
\end{pmatrix} \ \mid \ a,b,c,d,e,f \in F \right \} \oplus \left \{ \begin{pmatrix} 
\alpha a' & \alpha e' & 0 & 0 \\
0 & \alpha b' & 0 & 0 \\
0 & 0 & \alpha c' & 0 \\
0 & 0 & \alpha f' & \alpha d' 
\end{pmatrix} \ \mid \ a',b',c',d',e',f' \in F \right \}. 
\vspace{0.2 cm}
$$
We denote by $M_i$ the superalgebra $M$ with the pseudoautomorphism $\rho_i$ given by
$$
\rho_i \left(
\begin{pmatrix} 
a+\alpha a' & e+\alpha e' & 0 & 0 \\
0 & b+\alpha b' & 0 & 0 \\
0 & 0 & c+\alpha c' & 0 \\
0 & 0 & f+\alpha f' & d+\alpha d' 
\end{pmatrix} \right) = \begin{pmatrix} 
d+ i \alpha d' & f+ i \alpha f' & 0 & 0 \\
0 & c+ i \alpha c' & 0 & 0 \\
0 & 0 & b+ i \alpha b' & 0 \\
0 & 0 & e+ i \alpha e' & a+ i \alpha a' 
\end{pmatrix}.
\vspace{0.2 cm}
$$
Analogously, $M_{-i}$ denote the superalgebra $M$ with pseudoautomorphism $\rho_{-i}$ defined as $\rho_i$ but with $-i$ instead of $i$.

The last $p$-algebras we have to consider are the following:
\begin{itemize}
    \item $C_{27} $ is the subalgebra of $M_i$ with $a=d$, $b=c$, $a' = d'$ and $c'=b'=0$;
\vspace{0.1 cm}

\item $C_{28} $ is the subalgebra of $M_{-i}$ with $a=d$, $b=c$, $a' = d'$ and $c'=b'=0$;
\vspace{0.1 cm}

    \item $C_{29} $ is the subalgebra of $M_i$ with $a=d$, $b=c$, $b' = c'$ and $a'=d'=0$;
\vspace{0.1 cm}

    \item $C_{30} $ is the subalgebra of $M_{-i}$ with $a=d$, $b=c$, $b' = c'$ and $a'=d'=0$.
\vspace{0.1 cm}
\end{itemize}

Using the same approach as in Remark \ref{exponent of C5 ... C20}, we get the following result.

\begin{Remark} \label{exponent of C27 ... C30}
For $j = 23, \ldots, 30$, we have that $\exp^p(C_j) = 3$.
\end{Remark} 

Now we can prove the following lemma.

\begin{Lemma}  \label{C21 ... C24 in var(A)}
Assume that the field $F$ is also algebraically closed. 
Let $A = B_1 \oplus \cdots \oplus B_k + J$ be a finite dimensional $p$-algebra over $F$ such that $B_l J B_m \ne 0$ with 
$(B_l, B_m) \in  \left \{ (F, D), (D, F), (F, D^i), (D^i, F), (F, D^{-i}), (D^{-i}, F) \right \}$, $l \ne m$.
Then $C_j \in \V^p(A)$, for some $j \in \left \{ 23, \ldots, 30 \right \}$.
\end{Lemma}
\begin{proof}
Suppose first that $(B_l, B_m) =  (D, F)$. Let $e_l = e_1 + e_2$ and $e_m = e_3$ be the unit elements of $B_{l} $ and $ B_{m} $, respectively. Clearly $p(e_1) = e_2 $ and $p(e_3) = e_3$. Since $B_{l} J B_{m} \neq 0$, there exists a homogeneous element $j \in J$ such that 
$$
e_l j e_m = (e_1+e_2)je_3 \neq 0.
$$
Without loss of generality, we may assume that $e_1 j e_3 \ne 0$. Let $U$ be the $p$-algebra linearly generated by the elements 
$$
e_1, \ \ \ e_2, \ \ \ e_3, \ \ \  e_1 j e_3, \ \ \  e_2 p(j) e_3.
$$
When the homogeneous degree of $j$ is $0$, the map $f$ defined by
$$
f(e_{1}) = e_{1,1}, \ \ \ \
f(e_{2}) = e_{4,4}, \ \ \ \
f(e_{3}) = e_{2,2} + e_{3,3}, \ \ \ \
f(e_{1} j e_3) = e_{1,2}, \ \ \ \
f(e_{2} p(j) e_3) = e_{4,3},
$$
is an isomorphism of $p$-algebras between $U$ and $C_{23}$. When the homogeneous degree of $j$ is $1$, the map $f$ defined by
$$
f(e_{1}) = e_{1,1}, \ \ \ \
f(e_{2}) = e_{4,4}, \ \ \ \
f(e_{3}) = e_{2,2} + e_{3,3}, \ \ \ \
f(e_{1} j e_3) = e_{1,2}, \ \ \ \
f(e_{2} p(j) e_3) = i e_{4,3},
$$
is an isomorphism of $p$-algebras between $U$ and $C_{25}$.

In a similar way, when $(B_l, B_m) =  (F, D)$ we obtain that $C_{24}, C_{26} \in \V^p(A)$. 

Suppose now that $(B_l, B_m) =  (D^i, F)$. Let $e_l = e_1 $ and $e_m = e_2$ be the unit elements of $B_{l} $ and $ B_{m} $, respectively. Clearly $p(e_1) = e_1 $ and $p(e_2) = e_2$. Since $B_{l} J B_{m} \neq 0$, there exists a homogeneous element $j \in J$ such that 
$$
e_l j e_m = e_1je_2 \neq 0.
$$
Let $U$ be the $p$-algebra linearly generated by the elements 
$$
e_1, \ \ \ e_2, \ \ \ ce_1, \ \ \  e_1 j e_2, \ \ \  e_1 p(j) e_2, \ \ \  ce_1 j e_2, \ \ \  ce_1 p(j) e_2.
$$
Then $U$ is isomorphic to $C_{27}$ as $p$-algebras, via $f$ defined by
\begin{align*}
&f(e_1) = e_{1,1} + e_{4,4}, \ \ \ \
f(e_2) = e_{2,2} + e_{3,3}, \ \ \ \
f(ce_1) = c(e_{1,1} + e_{4,4}), \ \ \ \
f(e_1 j e_2) = e_{1,2}, \\
&f(e_1 p(j) e_2) = e_{4,3}, \ \ \ \ 
f(ce_1 j e_2) = ce_{1,2}, \ \ \ \ \ 
f(ce_1 p(j) e_2) = ce_{4,3},
\end{align*}
when the homogeneous degree of $j$ is $0$, and $f$ defined by
\begin{align*}
&f(e_1) = e_{1,1} + e_{4,4}, \ \ \ \
f(e_2) = e_{2,2} + e_{3,3}, \ \ \ \
f(ce_1) = c(e_{1,1} + e_{4,4}), \ \ \ \
f(e_1 j e_2) = ce_{1,2}, \\
&f(e_1 p(j) e_2) = ice_{4,3}, \ \ \ \ 
f(ce_1 j e_2) = e_{1,2}, \ \ \ \ 
f(ce_1 p(j) e_2) = ie_{4,3},
\end{align*}
when the homogeneous degree of $j$ is $1$.

In a similar way, when $(B_l, B_m) =  (D^{-i}, F)$, $(F, D^i)$ or $(F, D^{-i})$, we obtain that $C_{28}$, $C_{29}$ or $C_{30} \in \V^p(A)$, respectively.
\end{proof}

The following proposition proves that the list of $p$-algebras $C_1, \ldots, C_{30}$ cannot be reduced.

\begin{Proposition} \label{list cannot be reduced}
For all $l, j \in \left \{ 1, \ldots, 30 \right \}	$, $l \neq j$, $\Id^p(C_l) \nsubseteq \Id^p(C_j)$.
\end{Proposition}
\begin{proof}
By the classification of pseudoautomorphisms on $UT_n$ given in \cite{IoppoloMartino}, we have that the $p$-algebras $C_i$, $i = 7, \ldots, 22$ are not pairwise equivalent. Now, let $y$ denote an even variable, $z$ an odd one and $x$ any variable. The proof is completed by putting together the following facts.
\begin{itemize}

    \vspace{0.1 cm}
    \item $\Id^{p}(C_l) \not\subseteq \Id^{p}(C_j)$, $l=1, 3, 4, 7, 11, 13, 15, 16, 19, 22$, $j \neq l$: in fact $y^- \equiv 0$ on $C_l$ but not on $C_j$.
    \vspace{0.1 cm}



    \item $\Id^{p}(C_l) \not\subseteq \Id^{p}(C_j)$, $l=1,2, 7, 8, 9, 10, 23,24$, $j \neq l$: in fact $z \equiv 0$ on $C_l$ but not on $C_j$.
    \vspace{0.1 cm}

    \item $\Id^{p}(C_l) \not\subseteq \Id^{p}(C_j)$, $l=3$, $j = 4$: in fact $z^- \equiv 0$ on $C_l$ but not on $C_j$.
    \vspace{0.1 cm}

    \item $\Id^{p}(C_l) \not\subseteq \Id^{p}(C_j)$, $l=4$, $j = 3$: in fact $z^+ \equiv 0$ on $C_l$ but not on $C_j$.
    \vspace{0.1 cm}


    \item $\Id^{p}(C_l) \not\subseteq \Id^{p}(C_j)$, $l=2, \ldots, 6, 12, 14, 17, 18, 20, 21, 25, 26$, $j \ne l$: in fact $[y_1^+, y_2^+] \equiv 0$ on $C_l$ but not on $C_j$.
    \vspace{0.1 cm}

    \item $\Id^{p}(C_l) \not\subseteq \Id^{p}(C_j)$, $l=1, \ldots, 5$, $j= 7, \ldots, 22$,: in fact $\left[ [x_1, x_2]^2, x_3 \right] \equiv 0$ on $C_l$ but not on $C_j$.
    \vspace{0.1 cm}

    \item $\Id^{p}(C_l) \not\subseteq \Id^{p}(C_j)$, $l=5$, $j= 2, 3, 4, 25, 26$: in fact $[y^+, x] \equiv 0$ on $C_l$ but not on $C_j$.
    \vspace{0.1 cm}

    \item $\Id^{p}(C_l) \not\subseteq \Id^{p}(C_j)$, $l=5$, $j= 6$: in fact $y^- z^- + z^- y^- \equiv 0$ on $C_l$ but not on $C_j$.
    \vspace{0.1 cm}


    \item $\Id^{p}(C_l) \not\subseteq \Id^{p}(C_j)$, $l=6$, $j=1, \ldots, 5, 7 \ldots, 30$: in fact $[x_1, x_2] \equiv 0 $ on $C_l$ but not on $C_j$.
    \vspace{0.1 cm}







    \item $\Id^{p}(C_l) \not\subseteq \Id^{p}(C_j)$, $l=7, \ldots, 30$, $j=1, \ldots, 6$: in fact $ \exp^p(C_l) < \exp^p(C_j)$.
    \vspace{0.1 cm}

    \item $\Id^{p}(C_l) \not\subseteq \Id^{p}(C_j)$, $l=6, 23, \ldots, 30$, $j=1, \ldots, 5, 7, \ldots, 22$: in fact $ [x_1,x_2][x_3,x_4] \equiv 0$ on $C_l$ but not on $C_j$.
    \vspace{0.1 cm}


    \item $\Id^{p}(C_l) \not\subseteq \Id^{p}(C_j)$, $l=7, \ldots, 22$, $j=23,\ldots, 26$: in fact $ y_1^- y_2^- y_3^- \equiv 0$ on $C_l$ but not on $C_j$.
    \vspace{0.1 cm}







        \item $\Id^{p}(C_l) \not\subseteq \Id^{p}(C_j)$, $l=23$, $j=24$: in fact $[y_1^-, y_2^-]y_3^- \equiv 0 $ on $C_l$ but not on $C_j$.
    \vspace{0.1 cm}

    \item $\Id^{p}(C_l) \not\subseteq \Id^{p}(C_j)$, $l=24$, $j=23$: in fact $y_3^- [y_1^-, y_2^-] \equiv 0$ on $C_l$ but not on $C_j$.
    \vspace{0.1 cm}


    \item $\Id^{p}(C_l) \not\subseteq \Id^{p}(C_j)$, $l=25$, $j=26$: in fact $z^+ y^- \equiv 0$ on $C_l$ but not on $C_j$.
    \vspace{0.1 cm}

    \item $\Id^{p}(C_l) \not\subseteq \Id^{p}(C_j)$, $l=26$, $j=25$: in fact $ y^- z^+ \equiv 0$ on $C_l$ but not on $C_j$.
    \vspace{0.1 cm}

    \item $\Id^{p}(C_l) \not\subseteq \Id^{p}(C_j)$, $l=27, \ldots, 30$, $j=23, \ldots, 26$: in fact $ y_1^- y_2^- \equiv 0$ on $C_l$ but not on $C_j$.
    \vspace{0.1 cm}

    \item $\Id^{p}(C_l) \not\subseteq \Id^{p}(C_j)$, $l=28, 30$, $j=27, 29$: in fact $ z_1^+ z_2^+ \equiv 0$ on $C_l$ but not on $C_j$.
    \vspace{0.1 cm}

    \item $\Id^{p}(C_l) \not\subseteq \Id^{p}(C_j)$, $l=27, 29$, $j=28, 30$: in fact $ z_1^- z_2^- \equiv 0$ on $C_l$ but not on $C_j$.
    \vspace{0.1 cm}

    \item $\Id^{p}(C_l) \not\subseteq \Id^{p}(C_j)$, $l=27, 30$, $j=28, 29$: in fact $ z^- z^+ \equiv 0$ on $C_l$ but not on $C_j$.
    \vspace{0.1 cm}

    \item $\Id^{p}(C_l) \not\subseteq \Id^{p}(C_j)$, $l=28, 29$, $j=27, 30$: in fact $ z^+ z^- \equiv 0$ on $C_l$ but not on $C_j$.
  
\end{itemize}
\end{proof}

Now we are in a position to characterize the $p$-algebras $A$ with $\exp^p(A) \leq 2$.

\begin{Theorem} \label{characterization algebras with exponent greater than two}
Let $A$ be a finite dimensional $p$-algebra over $F$. Then $\exp^p(A) \leq 2 $ if and only if $C_j \notin \V^p(A)$, for any $j \in 	\left \{ 1, \ldots, 30 \right \}	$.
\end{Theorem}
\begin{proof}
Since we are dealing with $p$-codimensions that do not change by extending the base field, in what follows we may assume that the field $F$ is algebraically closed.

First let $\exp^p(A) \leq 2 $. Since $\exp^p(C_i)>2$, by Remarks \ref{exponent of C1 C2 C3 C4}, \ref{exponent of C5 ... C20} and \ref{exponent of C27 ... C30}, we get  $C_j \notin \V^p(A)$,  $j \in 	\left \{ 1, \ldots, 30 \right \}$.

Conversely, let $C_j \notin \V^p(A)$, for any $j \in \left \{ 1, \ldots, 30 \right \}$.   Hence by Theorem \ref{WM}  we can write $A = B_1 \oplus \cdots \oplus B_m + J$, where the $B_j$'s are simple $p$-algebras isomorphic to those ones given in Theorem \ref{phi super simple algebras}. Since $C_1, \ldots, C_6 \not \in \V^p(A)$, according to Lemma \ref{remark}, we have that $\dim_F B_l < 4$, for any $l$.

Suppose by contradiction that $\exp^p(A) > 2$. Then by Theorem \ref{the exponent}, one of the following possibilities occurs:
\begin{itemize}
\item[1.] there exist distinct $B_{i_1}, B_{i_2}, B_{i_3}$ such that $B_{i_1} J B_{i_2} J B_{i_3} \neq 0$ and $B_{i_1} \cong B_{i_2} \cong B_{i_3} \cong F$,
\vspace{0.1 cm}

\item[2.] for some $i_1 \neq i_2$, $B_{i_1} J B_{i_2} \neq 0$ and $B_{i_1} \cong F $ and $ B_{i_2} \cong D$ or $D^i$ or $D^{-i}$,
\vspace{0.1 cm}

\item[3.] for some $i_1 \neq i_2$, $B_{i_1} J B_{i_2} \neq 0$ and $B_{i_1} \cong D $ or $D^i$ or $D^{-i}$ and $ B_{i_2} \cong F$.
\vspace{0.1 cm}


\end{itemize}

Notice that when we are in the situation $B_{i_1} J B_{i_2} \ne 0$ with $B_{i_1} \cong D $ or $D^i$ or $D^{-i}$ and $ B_{i_2} \cong D$ or $D^i$ or $D^{-i}$, it easily follows that one of the last two cases occurs. 

If $1.$ holds, then, by Lemma \ref{C5 ... C20 in var(A)}, $C_j \in \V^p(A)$, for some $j \in \left \{ 7, \ldots, 22 \right \}$, a contradiction.
We reach a contradiction also in all the other cases, since by Lemma \ref{C21 ... C24 in var(A)}, we should have that $C_j \in \V^p(A)$, for some $j \in \left \{ 23,\ldots, 30 \right \}$. 
\end{proof}

In light of Theorems \ref{cn(A) pol bounded} and \ref{characterization algebras with exponent greater than two}, we get the characterization of $p$-algebras with $p$-exponent equal to two.

\begin{Corollary} \label{characterization algebras with exponent two}
Let $A$ be a finite dimensional $p$-algebra over $F$. Then $\exp^p(A) = 2 $ if and only if 
\begin{itemize}
    \item[-] $C_j \not \in \V^p(A)$, for all $j \in 	\left \{ 1, \ldots, 30 \right \}$ and
    \vspace{0.1 cm}
    \item[-] either $UT_2$ or $UT_2^{-}$ or $UT_2^{i}$ or $UT_2^{-i}$ or $D$ or $D^i$ or $D^{-i} \in \V^p(A)$.
\end{itemize}
\end{Corollary}

Now let us  slightly change the definition of minimal varieties given at the end of Section \ref{Section 3}: a variety  $\mathcal V$ of $p$-algebras
is minimal with respect to the $p$-exponent if for any proper subvariety  $\mathcal U,$ generated by a  finite dimensional $p$-algebra, 
we have that $\exp^{p}(\mathcal{V})>\exp^{p}(\mathcal{U}).$ 
 By using this definition we get the following.
\begin{Corollary} \quad
\begin{itemize}
    \item[1.] The $p$-algebras $ C_j$,  $j=1, \ldots, 6, $  generate minimal varieties of $p$-exponent $4$.
    \vspace{0.1 cm}
    \item[2.] The $p$-algebras $ C_j$,  $j=7, \ldots, 30, $ are the only finite dimensional algebras, up to equivalence, generating minimal varieties of $p$-exponent $3$.
    \end{itemize}
    \end{Corollary} 
    \begin{proof} We prove just item $2.$ (the proof of $1.$ is similar).
    
Let $\mathcal{V}$ be a proper subvariety of $\V^{p}(C_j)$, $j=7, \ldots, 30$. Clearly $C_j \notin \mathcal{V}$. Also, by Proposition \ref{list cannot be reduced}, we get that  $C_l \notin \mathcal{V},$ for any $l=1, \ldots, 30.$
Then, from Theorem \ref{characterization algebras with exponent greater than two}, 
$\exp^{p}(\mathcal{V})\le 2$ and we are done.

Now suppose that there exists a minimal variety $\mathcal{U}$ of $p$-exponent $3$ which is not generated by any of the algebras in 2. Since $\mathcal{U}$ is minimal and its  $p$-exponent is $3$, $C_j\notin \mathcal{U},$ for any $j.$ Then by Theorem \ref{characterization algebras with exponent greater than two} we should have $\exp^{p}(\mathcal{U})\leq 2,$ a contradiction.
\end{proof}

\end{document}